\documentclass[11pt]{amsart}

\usepackage[a4paper, total={6.5in, 9.2in}]{geometry}

\usepackage[english]{babel}

\usepackage{amsmath,amsthm,amsfonts,amssymb}

\usepackage[all]{xy}

\usepackage{color}

\usepackage{xparse}

\newcommand{\Dotfill}{\leaders\hbox to 8pt{\hss.\hss}\hfill}

\usepackage{tikz}
\tikzset{
v/.style={draw, fill, circle, minimum size=1.5mm, inner sep=0},
b/.style={draw , circle, minimum size=2.5mm, inner sep=.5mm},
e/.style={very thick},
vs/.style={draw, fill, circle, minimum size=1mm, inner sep=0},
bs/.style={draw, fill, circle, minimum size=1.5mm, inner sep=0mm},
es/.style={thick}
}
\newlength{\nodeheight}
\newlength{\nodewidth}
\usetikzlibrary{arrows,matrix,positioning}

\usepackage[colorlinks, linkcolor=blue,citecolor=blue,urlcolor=blue]{hyperref}
\usepackage{aliascnt}

\numberwithin{thmcounter}{section}
\newaliascnt{thmauto}{thmcounter}

\newaliascnt{Defauto}{thmcounter}

\newaliascnt{exauto}{thmcounter}

\newaliascnt{lemauto}{thmcounter}

\newaliascnt{propauto}{thmcounter}

\newaliascnt{corauto}{thmcounter}

\newaliascnt{remauto}{thmcounter}

\newaliascnt{convauto}{thmcounter}

\newtheorem{thmA}{Theorem}

\newtheorem{thm}[thmauto]{Theorem}
\newtheorem{prop}[propauto]{Proposition}
\newtheorem{lem}[lemauto]{Lemma}
\newtheorem{cor}[corauto]{Corollary}

\theoremstyle{definition}
\newtheorem{Def}[Defauto]{Definition}
\newtheorem{rem}[remauto]{Remark}

\numberwithin{equation}{section}

\DeclareMathOperator{\Hom}{Hom}
\DeclareMathOperator{\End}{End}

\providecommand{\id}{\ensuremath\mathrm{id}}

\DeclareMathOperator{\GL}{GL}
\DeclareMathOperator{\Orth}{O}

\DeclareMathOperator{\Fun}{Fun}

\newcommand{\hooklongrightarrow}{\lhook\joinrel\longrightarrow}

\newcommand{\inject}{\hooklongrightarrow}

\newcommand{\tens}[1][]{\mathbin{\otimes_{#1}}}

\DeclareMathOperator{\Ind}{Ind}
\DeclareMathOperator{\Res}{Res}

\DeclareMathOperator{\TL}{TL}

\DeclareMathOperator{\Br}{Br}

\providecommand{\CTL}{\ensuremath\mathcal C_{\TL}}

\providecommand{\CBr}{\ensuremath\mathcal C_{\Br}}
\providecommand{\CP}{\ensuremath\mathcal C_{P}}
\providecommand{\CA}{\ensuremath{\mathcal{C}_A}}
\providecommand{\FI}{\ensuremath\mathsf{FI}}
\providecommand{\xmod}[1]{\ensuremath{#1\mathsf{-mod}}}
\providecommand{\CAmod}{\ensuremath{\CA\mathsf{-mod}}}

\title{Representation stability for diagram algebras}
\author{Peter Patzt}
\thanks{Peter Patzt was supported by the Danish National Research Foundation through the Copenhagen Centre for Geometry and Topology (DNRF151)}
\address{Centre for Geometry and Topology, University of Copenhagen}
\email{ppatzt@gmail.com}

\begin{document}

\maketitle
\begin{abstract}We introduce stability categories for diagram algebras---analogues to Randal-Williams and Wahl's homogeneous categories. We use these to study representation stability properties of the Temperley--Lieb algebras, the Brauer algebras, and the partition algebras.\end{abstract}
    
%

\section{Introduction}

Representation stability is the study of sequences of representations. In the past, these were group representation of sequences of groups first and foremost the symmetric groups studied originally by Church--Ellenberg--Farb \cite{CEF}. A variety of different sequences of groups followed, see for example \cite{Wi,PS,PaTorelli,Nate}. The technique is always to construct a suitable category that describes the action of the groups and the interactions between them, and then study the representations over such category. Many of these constructions fit into common frameworks as those by Randal-Williams--Wahl \cite{RW} and others \cite{PS,Hepworth,PaCS}. Representation stability expanded to also study the representations over categories that do not come from sequences of groups but still have close resemblance for example in \cite{RamosFId,NSSFIM,PSSOI,NagpalVI}.

In this paper, we will study sequences of representations over the Temperley--Lieb algebras, the Brauer algebras, and the partition algebras. In order to do so, we construct a category enriched over $R$-modules and consider its representations (linear functors to the category of $R$-modules). The main theorems are that representations presented in finite degree exhibit representation stability, i.e.\ the decomposition into irreducible representations stabilizes.

We will deal with the following diagram algebras: the Temperley--Lieb algebra---first introduced by Temperley--Lieb \cite{TL} in the context of statistical mechanics---, the Brauer algebra---introduced by Brauer \cite{B} to study representations of the orthogonal and the symplectic groups---, and the partition algebra---introduced by Martin \cite{M91} also in the context of statistical mechanics. These (and other) diagram algebras have a distinguished basis that one interprets as diagrams. For these and other reasons they behave similar to group algebras.

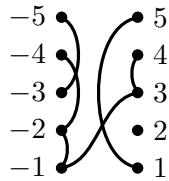
\begin{figure}[h!]
\centering
\begin{tikzpicture}
\fill (0,0)           circle (.75mm) node[left=2pt](a5) {$-5$};
\fill (0,\nodeheight) circle (.75mm) node[left=2pt](a4) {$-4$};
\multiply\nodeheight by 2
\fill (0,\nodeheight) circle (.75mm) node[left=2pt](a3) {$-3$};
\divide\nodeheight by 2
\multiply\nodeheight by 3
\fill (0,\nodeheight) circle (.75mm) node[left=2pt](a2) {$-2$};
\divide\nodeheight by 3
\multiply\nodeheight by 4
\fill (0,\nodeheight) circle (.75mm) node[left=2pt](a1) {$-1$};
\divide\nodeheight by 4

\fill (\nodewidth,0) circle (.75mm) node[right=2pt](b5) {$5$};
\fill (\nodewidth,\nodeheight) circle (.75mm) node[right=2pt](b4) {$4$};
\multiply\nodeheight by 2
\fill (\nodewidth,\nodeheight) circle (.75mm) node[right=2pt](b3) {$3$};
\divide\nodeheight by 2
\multiply\nodeheight by 3
\fill (\nodewidth,\nodeheight)  circle (.75mm) node[right=2pt](b2) {$2$};
\divide\nodeheight by 3
\multiply\nodeheight by 4
\fill (\nodewidth,\nodeheight)  circle (.75mm) node[right=2pt](b1) {$1$};
\divide\nodeheight by 4

\draw[e] (a1) to[out=0, in=180] (b3);
\draw[e] (a1) to[out=0, in=0] (a2);
\draw[e] (a2) to[out=0, in=0]   (a4);
\draw[e] (b3) to[out=180, in=180] (b4);
\draw[e] (a5) to[out=0, in=0] (a3);
\draw[e] (b1) to[out=180,in=180] (b5);

\end{tikzpicture}
\caption{Visualization of the partition $\{\{-5,-3\},\{-4,-2,-1,3,4\},\{1,5\},\{2\}\}$}
\label{fig:P5ex}
\end{figure}

The partition algebra $P_n = P_n(R,\delta)$ is a free $R$-module over the basis of all partitions of the union of the sets $[-n]=\{-n,\dots,-1\}$ and $[n]= \{1,\dots,n\}$. We visualize these by placing $n$ dots on the left for the negative numbers and $n$ dots on the right for the positive numbers and connect two dots if they are in the same block. (We do not connect all dots in the same block with each other, but rather rely on transitivity.) Given two partition $p$ and $r$, to calculate their product $pr$, we place $p$ to the left of $r$ so that the dots align and connect the blocks of $p$ with the one of $r$. This way we get a new partition $s$ of $[-n]\cup [n]$. Assume there are $a$ blocks that connect solely to dots in the middle, then we define $pr = \delta^a s$.

\begin{figure}[h!]
\centering
\[
\begin{tikzpicture}[x=1.5cm,y=-.5cm,baseline=-1.05cm]

\node[v] (a1) at (0,0) {};
\node[v] (a2) at (0,1) {};
\node[v] (a3) at (0,2) {};
\node[v] (a4) at (0,3) {};
\node[v] (a5) at (0,4) {};

\node[v] (b1) at (1,0) {};
\node[v] (b2) at (1,1) {};
\node[v] (b3) at (1,2) {};
\node[v] (b4) at (1,3) {};
\node[v] (b5) at (1,4) {};

\draw[e] (a1) to[out=0, in=180] (b3);
\draw[e] (a1) to[out=0, in=0] (a2);
\draw[e] (a2) to[out=0, in=0]   (a4);
\draw[e] (b3) to[out=180, in=180] (b4);
\draw[e] (a5) to[out=0, in=0] (a3);
\draw[e] (b1) to[out=180,in=180] (b5);

\end{tikzpicture}
\quad
\cdot
\quad
\begin{tikzpicture}[x=1.5cm,y=-.5cm,baseline=-1.05cm]

\node[v] (b1) at (0,0) {};
\node[v] (b2) at (0,1) {};
\node[v] (b3) at (0,2) {};
\node[v] (b4) at (0,3) {};
\node[v] (b5) at (0,4) {};

\node[v] (c1) at (1,0) {};
\node[v] (c2) at (1,1) {};
\node[v] (c3) at (1,2) {};
\node[v] (c4) at (1,3) {};
\node[v] (c5) at (1,4) {};

\draw[e] (b2) to[out=0,in=0] (b5);
\draw[e] (b3) to[out=0,in=180] (c1);
\draw[e] (b4) to[out=0,in=180] (c5);
\draw[e] (c1) to[out=180,in=180] (c2);
\draw[e] (c3) to[out=180,in=180] (c4);

\end{tikzpicture}
\quad
=
\quad
\begin{tikzpicture}[x=1.5cm,y=-.5cm,baseline=-1.05cm]

\node[v] (a1) at (0,0) {};
\node[v] (a2) at (0,1) {};
\node[v] (a3) at (0,2) {};
\node[v] (a4) at (0,3) {};
\node[v] (a5) at (0,4) {};

\node[v] (b1) at (1,0) {};
\node[v] (b2) at (1,1) {};
\node[v] (b3) at (1,2) {};
\node[v] (b4) at (1,3) {};
\node[v] (b5) at (1,4) {};

\node[v] (c1) at (2,0) {};
\node[v] (c2) at (2,1) {};
\node[v] (c3) at (2,2) {};
\node[v] (c4) at (2,3) {};
\node[v] (c5) at (2,4) {};

\draw[e] (a1) to[out=0, in=180] (b3);
\draw[e] (a1) to[out=0, in=0] (a2);
\draw[e] (a2) to[out=0, in=0]   (a4);
\draw[e] (b3) to[out=180, in=180] (b4);
\draw[e] (a5) to[out=0, in=0] (a3);
\draw[e] (b1) to[out=180,in=180] (b5);

\draw[e] (b2) to[out=0,in=0] (b5);
\draw[e] (b3) to[out=0,in=180] (c1);
\draw[e] (b4) to[out=0,in=180] (c5);
\draw[e] (c1) to[out=180,in=180] (c2);
\draw[e] (c3) to[out=180,in=180] (c4);

\end{tikzpicture}
\quad
= \delta \cdot
\quad
\begin{tikzpicture}[x=1.5cm,y=-.5cm,baseline=-1.05cm]

\node[v] (a1) at (0,0) {};
\node[v] (a2) at (0,1) {};
\node[v] (a3) at (0,2) {};
\node[v] (a4) at (0,3) {};
\node[v] (a5) at (0,4) {};

\node[v] (c1) at (1,0) {};
\node[v] (c2) at (1,1) {};
\node[v] (c3) at (1,2) {};
\node[v] (c4) at (1,3) {};
\node[v] (c5) at (1,4) {};

\draw[e] (a1) to[out=0, in=180] (c1);
\draw[e] (a1) to[out=0, in=0] (a2);
\draw[e] (a2) to[out=0, in=0]   (a4);
\draw[e] (a5) to[out=0, in=0] (a3);
\draw[e] (c1) to[out=180,in=180] (c2);
\draw[e] (c1) to[out=180,in=180] (c2);
\draw[e] (c3) to[out=180,in=180] (c4);
\draw[e] (c2) to[out=180,in=180] (c5);
\end{tikzpicture}
\]
\caption{Multiplication in the partition algebra}
\label{fig:MultP5}
\end{figure}
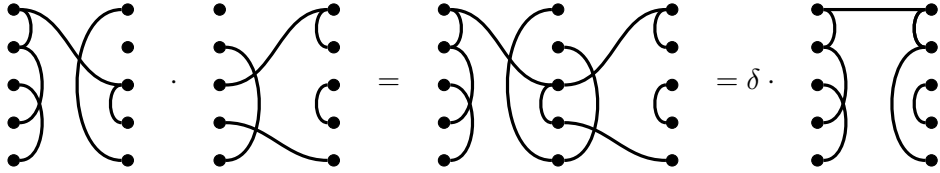

The Brauer algebra $\Br_n = \Br_n(R,\delta) \subset P_n(R,\delta)$ is the restriction of the partition algebra to those partitions that are perfect matchings, i.e. all blocks have exactly two elements. The Temperley--Lieb algebra $\TL_n = \TL_n(R,\delta) \subset \Br_n(R,\delta)$ is the restriction of the Brauer algebra to those partitions that contain no two blocks $\{a,b\}$ and $\{c,d\}$ with $a<c<b<d$, i.e. the lines in the diagram do not intersect.

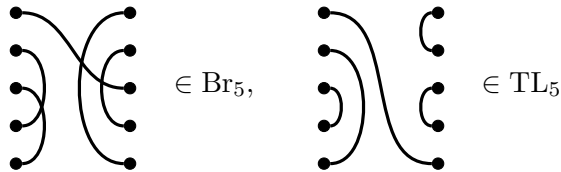
\begin{figure}[h!]
\centering
\[
\begin{tikzpicture}[x=1.5cm,y=-.5cm,baseline=-1.05cm]

\node[v] (a1) at (0,0) {};
\node[v] (a2) at (0,1) {};
\node[v] (a3) at (0,2) {};
\node[v] (a4) at (0,3) {};
\node[v] (a5) at (0,4) {};

\node[v] (b1) at (1,0) {};
\node[v] (b2) at (1,1) {};
\node[v] (b3) at (1,2) {};
\node[v] (b4) at (1,3) {};
\node[v] (b5) at (1,4) {};

\draw[e] (a1) to[out=0, in=180] (b3);
\draw[e] (a2) to[out=0, in=0]   (a4);
\draw[e] (b2) to[out=180, in=180] (b4);
\draw[e] (a5) to[out=0, in=0] (a3);
\draw[e] (b1) to[out=180,in=180] (b5);

\end{tikzpicture}
\quad
\in \Br_5,
\qquad
\begin{tikzpicture}[x=1.5cm,y=-.5cm,baseline=-1.05cm]

\node[v] (b1) at (0,0) {};
\node[v] (b2) at (0,1) {};
\node[v] (b3) at (0,2) {};
\node[v] (b4) at (0,3) {};
\node[v] (b5) at (0,4) {};

\node[v] (c1) at (1,0) {};
\node[v] (c2) at (1,1) {};
\node[v] (c3) at (1,2) {};
\node[v] (c4) at (1,3) {};
\node[v] (c5) at (1,4) {};

\draw[e] (b2) to[out=0,in=0] (b5);
\draw[e] (b3) to[out=0,in=0] (b4);
\draw[e] (b1) to[out=0,in=180] (c5);
\draw[e] (c1) to[out=180,in=180] (c2);
\draw[e] (c3) to[out=180,in=180] (c4);

\end{tikzpicture}
\quad
\in \TL_5\]
\caption{Examples of elements in $\Br_5$ and $\TL_5$}
\label{fig:Br5TL5ex}
\end{figure}

In this paper, we will generalize the construction of Randal-Williams--Wahl \cite{RW} to the Temperley--Lieb algebras, the Brauer algebras, and the partition algebras and define respective categories $\CTL$, $\CBr$, and $\CP$ in \autoref{thm:CA}.  We furthermore generalize the notion of representation stability defined by Church--Farb \cite{CF} to sequences of semisimple representations over these algebras in \autoref{def:repstab}. Our main theorems are then the following.

\begin{thmA}\label{thmA:TL} Under the assumption that $R = \mathbb C$ and $\delta^2 = 2+q+q^{-1}$ for some $q\in \mathbb C$,
a finitely presented $\CTL$-module is representation stable.
\end{thmA}

\begin{thmA}\label{thmA:Br} Under the assumption that $R = \mathbb C$ and $\delta\in \mathbb C\setminus\{0,1,2, \dots\}$,
a finitely presented $\CBr$-module is representation stable.
\end{thmA}

\begin{thmA}\label{thmA:P} Under the assumption that $R = \mathbb C$ and $\delta\in \mathbb C\setminus\{0,1,2, \dots\}$,
a finitely presented $\CP$-module is representation stable.
\end{thmA}

Categories of Brauer diagrams and Temperley--Lieb diagrams have been studied before. For example, Sam--Snowden study the representation theory of their upwards Brauer category \cite{SSStabilityPatterns} and the (whole) Brauer category \cite{LZBr,SSBr} in the context of Schur--Weyl duality and the Deligne category. These categories behave quite differently from the category $\CBr$ that we use to investigate stability phenomena of representations of Brauer algebras. For sequences of representations of Temperley--Lieb algebras, Sitaraman \cite{Sitaraman} phrases representation stability in terms of modules over his category $\mathrm{LS}$. Although $\mathrm{LS}$ is quite different from our $\CTL$, it captures a very similar stability phenomenon. In particular, the principal sequences of stable representations are analogous. To the best of our knowledge, his results are disjoint from our results and were independently derived. In particular, he only considers the context in which $\delta=1$ so that $\TL_n$ is a monoid algebra, whereas we only consider $\delta$ such that $\TL_n$ is split semisimple, which excludes $\delta =1$.

In work together with Boyd and Hepworth \cite{BHP}, we use a chain complex that is inspired by a construction of Randal-Williams--Wahl \cite{RW} on the symmetric monoidal category $\CBr$.


\vspace{1ex}

\noindent\textbf{Acknowledgements.} I would like to thank my advisor Holger Reich to propose the problem of representation stability for diagram algebras as part of my PhD. I would like to thank Rachael Boyd and Richard Hepworth for contacting me about their progress on homological stability of Temperly-Lieb algebras in \cite{BH}. This inspired me to finish this paper and it inspired our work \cite{BHP} on the homology of Brauer algebras. I would like to thank Reiner Hermann, Steffen Koenig, Jeremy Miller,  Steven Sam, Andrew Snowden, Maithreya Sitaraman, and Catharina Stroppel  for additional helpful conversations.

\section{Stability categories}

In this section, we will generalize the notion of $\FI$-modules from the symmetric groups to the Temperley--Lieb algebras, Brauer algebras, and partition algebras. To avoid repetition we will write $(A_n)_{n\in \mathbb N}$ to mean one of the three sequences $(\TL_n)_{n\in \mathbb N}$, $(\Br_n)_{n\in \mathbb N}$, and $(P_n)_{n\in \mathbb N}$.\footnote{The natural numbers $\mathbb N = \{0,1,2,\dots\}$ in this paper.} It is essential for our consideration that
\[ A_0 \inject A_1 \inject A_2 \inject \dots\]
is a sequence of algebra monomorphisms. More generally, we need the embedding
\begin{equation} A_a \otimes A_b \inject A_{a+b}.\label{eq:embedding}\end{equation}
This can be achieved by putting a diagram on $a$ nodes on top of a diagram on $b$ nodes. This map gives a \emph{left embedding} and \emph{right embedding}
\[ A_a \inject A_a\otimes A_b \inject A_{a+b}\quad\text{and}\quad A_b \inject A_a\otimes A_b \inject A_{a+b}\]
where a diagram $d$ is sent to $d\otimes 1$ and $1\otimes d$ respectively. 

In other words, there is a monoidal category we denote by $\TL$, $\Br$, $P$, respectively or $A$ to mean any of those three, whose objects are the nonnegative integers, the morphisms are given by
\[ \Hom_{A}(m,n) = \begin{cases} A_n &\text{if $m=n$}\\ 0 &\text{if $m\neq n$,}\end{cases}\]
and the underlying monoid is the addition on the nonnegative integers.


For the diagram algebras $\TL_n$, $\Br_n$, and $P_n$ defined over the ring $R$, there is a canonical choice of a ``trivial'' module $R$, on which all invertible diagrams (i.e.\ permutations) act trivially and all other diagrams annihilate $R$. 

\begin{thm}\label{thm:CA}
There are categories $\CTL$, $\CBr$, and  $\CP$ enriched in $R$-modules whose objects are the nonnegative integers and whose morphisms are given by
\begin{gather*}
\Hom_{\CTL}(m,n) = \TL_n \tens[\TL_{n-m}] R\\
\Hom_{\CBr}(m,n) = \Br_n \tens[\Br_{n-m}] R\\
\Hom_{\CP}(m,n) = P_n \tens[P_{n-m}] R
\end{gather*}
where $A_{n-m} \subset A_n$ by the right embedding if $m\le n$ and 
\[\Hom_{\CTL}(m,n) =\Hom_{\CBr}(m,n) =\Hom_{\CP}(m,n) = 0\]
if $m>n$. Denote the category associated to the sequence $(A_n)_{n\in \mathbb N}$ by $\CA$.
\end{thm}

\begin{proof}
We only have to define the composition $l\to m\to n$ if $l\le m \le n$, as it is zero in all other cases. We define it by
\begin{align*}
& \Hom_{\CA}(m,n)\otimes \Hom_{\CA}(l,m)\\
= & (A_n \tens[A_{n-m}] R) \otimes (A_m \tens[A_{m-l}] R)\\
\twoheadrightarrow & (A_n \tens[A_{n-m}] R) \tens[A_m] (A_m \tens[A_{m-l}] R)\\
\cong & A_n\tens[(A_{n-m} \otimes A_{m-l})] (R \otimes R)\\
\twoheadrightarrow & A_n \tens[A_{n-l}] R = \Hom_{\CA}(l,n).
\end{align*}
Associativity is a consequence of the associativity of the tensor product. Note that $\End_{\CA}(n) \cong A_n$ and that the unit of $A_n$ give the identity morphism $\id_n$.\end{proof}

\begin{prop}
The categories $\CBr$ and $\CP$ are symmetrical monoidal.
\end{prop}

\begin{proof}
For this proof, let $(A_n)_{n\in \mathbb N}$ only stand for $(\Br_n)_{n\in \mathbb N}$ and $(P_n)_{n\in \mathbb N}$. Let $s_{a,b}\colon a+b \to b+a = a+b$ be given by the permutation (diagram)
\[s_{a,b}(x) = \begin{cases} x+b &\text{if $x\le a$}\\ x-a & \text{if $x>a$.}\end{cases}\]
Note that $s$ is a symmetry of the monoidal category $A$. We will prove that $\CA$ is symmetric monoidal with the same underlying monoid (the addition of nonnegative numbers) and the same symmetry as $A$.

Given two maps $f\colon a\to a+b$ and $g\colon c \to c+d$, we need to define $f+g\colon a+c \to a+b+c+d$, show that this map is functorial and prove that the symmetry given behaves natrually with the monoidal structure on $\CA$, i.e. \[(g+f)\circ s_{a,c} = s_{a+b,c+d}\circ (f+g).\] All other properties of a symmetric monoidal structure like the hexagon identity carry over from $A$ directly. 

We will prove all these assertions only for the basis elements $f= \bar f \otimes 1\in A_{a+b}\otimes_{A_b} R$ with $\bar f\in A_{a+b}$ a diagram and $g= \bar g \otimes 1\in A_{c+d}\otimes_{A_d} R$ with $\bar g\in A_{c+d}$ a diagram. Then we extend to the general case linearly.

First let $\phi_{m,n}\colon m \to n$ be defined by $\id_n\otimes 1\in A_{n}\otimes_{A_{n-m}} R = \Hom_{\CA}(m,n)$. We then define $f+g\colon a+c \to a+b+c+d$ with $f,g$ from above by the composition
\[\xymatrix{ a+c \ar[d]_{\phi_{a+c,a+c+b+d}} \ar@{-->}[rr]^{f+g}&&a+b+c+d\\
a+c+b+d \ar[rr]^{\id_a+s_{c,b}+\id_d} &&a+b+c+d \ar[u]_{\bar f + \bar g}. }\]
We observe that our construction is well defined because
\[ \xymatrix{
a+c \ar[r]\ar[rd]_\alpha &a+c+b+d \ar[r]\ar[d]^{{\id}+x+y} &a+b+c+d \ar[r]^{\bar fx+ \bar gy} \ar[d]_{{\id}+x+{\id}+y} & a+b+c+d\\
& a+c+b+d \ar[r]  & a+b+c+d \ar[ru]_{\bar f+ \bar g}
}\]
commutes where
\[ \alpha = \begin{cases} \phi_{a+c, a+c+b+d} &\text{if $x\in A_b$ and $y\in A_d$ are invertible diagrams}\\ 0&\text{if $x\in A_b$ and $y\in A_d$ are noninvertible diagrams.}\end{cases}\]

Together with the two commuting diagrams
\[\xymatrix{
a \ar[r]^{\phi}\ar[rd]_{\phi} & a+b+c \ar[d]^{{\id}+s_{b,c}}\\
& a+c+b
}\quad\xymatrix{
a+b\ar[r]^{\phi} \ar[d]^{s_{a,b}} & a+b+c \ar[d]^{s_{a,b} + \id}\\
b+a\ar[r]^{\phi} & b+a+c
}\]
One can easily prove functoriality of the defined monoidal structure. (Although it involves a quite large diagram.)

Showing naturality, we take advantage of $s_{a,b} = s_{b,a}^{-1}$. From the outer square of the following commutative diagram we derive that the symmetry is natural.
\[\xymatrix{
a+c \ar[rrr]^{f+g}\ar[dddd]_{s_{a,c}}\ar[rd]^\phi &&& a+b+c+d\ar[dddd]^{s_{a+b,c+d}}\\
& a+c+b+d \ar[d] & a+b+c+d \ar[l]\ar[dd]\ar[ld]\ar[ru]^{\bar f+\bar g}\\
& c+a+b+d\ar[d]\ar[rd]\\
& c+a+d+b \ar[r] &  c+d+a+b\ar[rd]^{\bar g+\bar f}\\
c+a \ar[ru]^\phi\ar[rrr]_{g+f} &&& c+d+a+b
}\]
\end{proof}

It turns out that $\Hom_{\CA}(m,n)$ is a free $R$-module. In the following propositions, we describe a basis in terms of diagrams. Let us start with the Brauer algebras.

\begin{prop}\label{prop:HomBr}
$\Hom_{\CBr}(m,n) = \Br_n \otimes_{\Br_{n-m}} R$ is a free $R$-module and the set of all diagrams with $n$ dots on the LHS and $m$ dots and an ``$(n-m)$-blob'' on top of the $m$ dots on the RHS such that every dot is connected to exactly on other dot or to the blob and the blob is connected to exactly $(n-m)$ dots forms a basis. (See  \autoref{fig:Ex Br Hom(2,5)} for examples.)
\end{prop}

\begin{figure}[h!]
\centering
\[
\begin{tikzpicture}[x=1.5cm,y=-.5cm,baseline=-1.05cm]

\node[v] (a1) at (0,0) {};
\node[v] (a2) at (0,1) {};
\node[v] (a3) at (0,2) {};
\node[v] (a4) at (0,3) {};
\node[v] (a5) at (0,4) {};

\node[b] (b1) at (1,1) {$3$};
\node[v] (b4) at (1,3) {};
\node[v] (b5) at (1,4) {};

\draw[e] (a1) to[out=0, in=180] (b1);
\draw[e] (a2) to[out=0, in=0]   (a4);
\draw[e] (b1) to[out=180, in=180] (b4);
\draw[e] (a5) to[out=0, in=0] (a3);
\draw[e] (b1) to[out=180,in=180] (b5);

\end{tikzpicture}
\quad,\quad
\begin{tikzpicture}[x=1.5cm,y=-.5cm,baseline=-1.05cm]

\node[v] (a1) at (0,0) {};
\node[v] (a2) at (0,1) {};
\node[v] (a3) at (0,2) {};
\node[v] (a4) at (0,3) {};
\node[v] (a5) at (0,4) {};

\node[b] (b1) at (1,1) {$3$};
\node[v] (b4) at (1,3) {};
\node[v] (b5) at (1,4) {};

\draw[e] (a1) to[out=0, in=180] (b1);
\draw[e] (a2) to[out=0, in=180]   (b1);
\draw[e] (b1) to[out=180, in=0] (a4);
\draw[e] (a5) to[out=0, in=0] (a3);
\draw[e] (b4) to[out=180,in=180] (b5);

\end{tikzpicture}
\quad
\in \Hom_{\CBr}(2,5),
\qquad
\begin{tikzpicture}[x=1.5cm,y=-.5cm,baseline=-1.05cm]

\node[v] (a1) at (0,0) {};
\node[v] (a2) at (0,1) {};
\node[v] (a3) at (0,2) {};
\node[v] (a4) at (0,3) {};
\node[v] (a5) at (0,4) {};

\node[b] (b1) at (1,1) {$3$};
\node[v] (b4) at (1,3) {};
\node[v] (b5) at (1,4) {};

\draw[e] (a1) to[out=0, in=180] (b1);
\draw[e] (a2) to[out=0, in=0]   (a4);
\draw[e] (b1) to[out=210, in=270] (.5,1);
\draw[e] (b1) to[out=150, in=90] (.5,1);
\draw[e] (a5) to[out=0, in=0] (a3);
\draw[e] (b4) to[out=180,in=180] (b5);

\end{tikzpicture}
\quad
\not\in \Hom_{\CBr}(2,5)\]
\caption{Two examples and one non-example for the described diagrams}
\label{fig:Ex Br Hom(2,5)}
\end{figure}
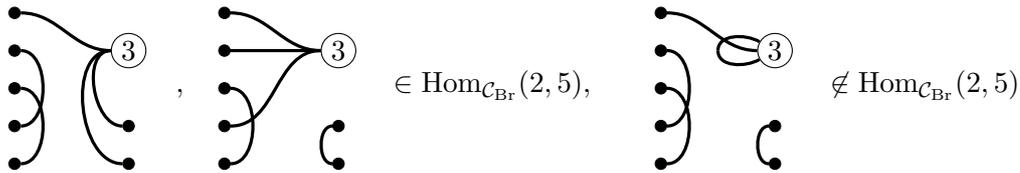

\begin{proof}
Define $J \subset \Br_n$ to be the free $R$-submodule generated by diagrams $b\in \Br_n$ that has a connection between two dots $m+1\le i,j\le n$. Note that $J$ is invariant under right multiplication by $\Br_{n-m}$.

We first observe that the image of $J\otimes_{\Br_{n-m}} R \to \Br_n \otimes_{\Br_{n-m}} R$ is zero. To do so, note that we can write every diagram $b\in J$ as $b = b'b''$, where $b'\in \Br_n$ and $b''\in \Br_{n-m} \subset \Br_n$ are diagrams and $b''$ is not invertible. This can be achieved by choosing a right-to-right connection $(i,j)$ with $m+1\le i,j \le n$ and a (thus existent) left-to-left connection $(i',j')$. Let $b'$ be the diagram in which we have the left-to-right connections $(i',i)$ and $(j',j)$ and all other connections the same as in $b$. Let $b''$ be the diagram that has the right-to-right connection $(i,j)$ and the left-to-left connection $(-i,-j)$ and all other connections are left-to-right $(-k,k)$.

Using right-exactness of the tensor product, we conclude that $\Br_n \otimes_{\Br_{n-m}} R \cong \Br_n/J \otimes_{\Br_{n-m}} R$. 

Next, define $I\subset \Br_{n-m}$ to be the free $R$-submodule generated by all noninvertible diagrams $b\in \Br_{n-m}$. $I$ annihilates both $R$ and $\Br_n/J$. That implies that the action of $\Br_{n-m}$ on $R$ and $\Br_n/J$ factors through $\Br_{n-m}/I \cong RS_{n-m}$ and thus $\Br_n/J \otimes_{\Br_{n-m}} R \cong \Br_n/J \otimes_{RS_{n-m}} R$.

Note that the images of diagrams $b\in\Br_n$ that have no connection between two dots $m+1\le i,j\le n$ form a basis of $\Br_n/J$. Because the symmetric group $S_{n-m}$ permutes the basis elements, $\Br_n/J \otimes_{RS_{n-m}} R$ is a free $R$-module and its basis is given by the orbits of the $S_{n-m}$-action on the basis of diagrams of $\Br_n/J$. This description coincides with the diagrammatic description asserted.
\end{proof}

For the Temperley--Lieb algebras we may simply restrict to those diagrams that have nonintersecting lines.

\begin{prop}
$\Hom_{\CTL}(m,n) = \TL_n \otimes_{\TL_{n-m}} R$ is a free $R$-module and the set of all diagrams with $n$ dots on the LHS and $m$ dots and an ``$(n-m)$-blob'' on top of the $m$ dots on the RHS such that every dot is connected to exactly on other dot or to the blob, the blob is connected to exactly $(n-m)$ dots, and non of the the connections intersect forms a basis.\end{prop}

For the partition algebras we must adjust the proof of  \autoref{prop:HomBr} slightly.

\begin{prop}\label{prop:HomP}
$\Hom_{\CP}(m,n) = P_n \otimes_{P_{n-m}} R$ is a free $R$-module and the set of all partitions of $[-n] \cup [m]$ with $(n-m)$ marked blocks forms a basis.
\end{prop}

\begin{rem}
One may view the $(n-m)$ marked blocks as connected to an $(n-m)$-blob but the blob does not connect the different blocks connected to it. (And non of the blocks connect to the blob can be empty.)
\end{rem}

\begin{proof}[Proof of  \autoref{prop:HomP}]
Let $J\subset P_n$ be the free $R$-submodule generated by partitions that contain a block that contains two elements $m+1\le i,j \le n$ or is a singleton $\{i\}$ with $m+1 \le i \le n$. 

As in the proof of  \autoref{prop:HomBr}, we will observe that the image of $J\otimes_{P_{n-m}} R \to P_n \otimes_{P_{n-m}} R$ is zero. We will write every partition $b\in J$ as a product $b = b'b''$ where $b'\in P_n$ and $b''\in P_{n-m}$ are both diagrams and $b''$ is not invertible. Let $b\in J$ be a partition and let $S$ be the intersection of a block and $\{m+1, \dots, n\}$ such that either $S$ is a singleton block or $S$ contains at least two elements. 

If the dot(s) in $S$ are not connected to at least one dot $k\le m$, pick another block $T$ in $b$ that doesn't intersect $\{m+1, \dots, n\}$. Let $b'$ be the same partition as $b$ except that $S$ and $T$ are joined to be one block and let $b''$ be the partition with blocks $S$ and $-S = \{ -l \mid l \in S\}$ and blocks $\{l,-l\}$ for all $l \not\in S$.

If the dots in $S$ are connected to at least one dot $k\le m$, then let $b'=b$ and let $b''$ be the partition with one block $S\cup -S$ and blocks $\{-l,l\}$ for all $l \not \in S$.

Using right-exactness of the tensor product, we conclude that $P_n \otimes_{P_{n-m}} R \cong P_n/J \otimes_{P_{n-m}} R$. 

Next, define $I\subset P_{n-m}$ to be the free $R$-submodule generated by all noninvertible diagrams $b\in P_{n-m}$. $I$ annihilates both $R$ and $P_n/J$. That implies that the action of $P_{n-m}$ on $R$ and $P_n/J$ factors through $P_{n-m}/I \cong RS_{n-m}$ and thus $P_n/J \otimes_{P_{n-m}} R \cong P_n/J \otimes_{RS_{n-m}} R$.

Note that the images of diagrams $b\in P_n$ that have no two connected dots $m+1\le i,j\le n$ and no singleton blocks $\{i\}$ with $m+1\le i\le n$ form a basis of $P_n/J$. Because the symmetric group $S_{n-m}$ permutes the basis elements, $P_n/J \otimes_{RS_{n-m}} R$ is a free $R$-module and its basis is given by the orbits of the $S_{n-m}$-action on the basis of diagrams of $P_n/J$. This description coincides with the diagrammatic description asserted.
\end{proof}

\begin{Def}
We call a linear functor $V\colon \CA \to \xmod{R}$  a \emph{$\CA$-module}. We denote $\Fun_{\mathrm{lin}}(\CA,\xmod{R})$ by $ \CAmod$.
\end{Def}

\begin{lem}\label{lem:Cmod criterion}
Let $(V_n, \phi_n)_{n \in \mathbb N}$ be a sequence of $A_n$-modules $V_n$ and $A_n$-homomorphisms $\phi_n \colon V_n \to V_{n+1}$. 

There is a (unique) $\CA$-module $V$ with $V(n) = V_n$ and $V(1_{A_{n+1}}\otimes_{A_1} 1_{R})  = \phi_n$ if and only if for all $m \le n$, diagrams $d\in A_{n-m}\subset A_n$ by right embedding, and $v\in V_m$
\[d\phi_{m,n} = \begin{cases} \phi_{m,n} &\text{if $d$ is an invertible diagram}\\ 0 &\text{if $d$ is a noninvertible diagram.}\end{cases}\]
\end{lem}

\begin{proof}
Assume $V$ is a $\CA$-module. In $\CA$ we have the equation
\[ d \cdot (1_{A_n}\tens[A_{n-m}] 1) = \begin{cases}1_{A_n}\tens[A_{n-m}] 1&\text{if $d$ is an invertible diagram}\\ 0 &\text{if $d$ is a noninvertible diagram.}\end{cases}\]
for every diagram $d\in a_{n-m}$.
Applying $V$, we get
\[d\phi_{m,n} = \begin{cases} \phi_{m,n} &\text{if $d$ is an invertible diagram}\\ 0 &\text{if $d$ is a noninvertible diagram.}\end{cases}\]
This proves the necessity of the condition.

Now assume the condition to be true. We define the functor $V$ by setting \[V(d\tens[{A_{n-m}}] 1)=d\phi_{m,n}\] for diagrams $d\in A_n$. The condition exactly tells us that this is well defined. We only need to check compatibility with composition. Let $d\in A_n$ and $d'\in A_m$be diagrams, then
\[ d\tens[A_{n-m}]1 \circ d'\tens[A_{m-l}]1 = dd' \tens[A_{n-l}] 1.\]
This is compatible with
\[ d\phi_{m,n} \circ d'\phi_{l,m} = dd' \phi_{l,n}.\]
For this note that $\phi_{m,n}$ is an $A_m$-homomorphism.
\end{proof}

\section{Representation stability}

Representation stability is a notion first introduced by Church--Farb \cite{CF} for the symmetric groups and the classical groups but always in a semisimple setting. We first recall the semisimple representation theory of the Temperley--Lieb algebras, the Brauer algebras, and the partition algebras. We then define what we mean by representation stability for $\CA$-modules. Sticking close to the ideas of Church--Ellenberg--Farb \cite{CEF}, we define the stability degree of $\CA$-modules and use it to prove that finitely presented $\CA$-modules are representation stable.

In this section, we will restrict to the base ring $R= \mathbb C$ the complex numbers although most statements can be made for any field of characteristic zero or even more generally.

\subsection{Semisimple representation theory of $\TL_n$, $\Br_n$, $P_n$}

Let us first recall the semisimple representation theory of the Temperley--Lieb algebras.

\begin{thm}[{\cite[Thm.\ 2.8.5]{GdHJ}}]\label{thm:TL repthy}
Let $P_0(X)=1$, $P_1(X)=1$, and $P_{k+1}(X) = P_k(X) -X\cdot P_{k-1}(X)$ define a sequence of polynomials in $\mathbb C[X]$. Let $\delta\in \mathbb C^{\times}$ be nonzero and $P_k(\delta^{-2}) \neq 0 $ for all $k\le n-1$.
Then $\TL_n(\delta,\mathbb C)$ is semisimple and decomposes into full matrix algebras over $\mathbb C$. Let $\Pi_{\TL_n}$ be the partitions of $n$ with at most two parts, then the  nonisomorphic simple $\TL_n$-modules are index over $\Pi_{\TL_n}$ and we denote them by $\TL_n(\lambda)$ for $\lambda\in \Pi_{\TL_n}$. 
\end{thm}

%

\begin{thm}[Branching rules]\label{thm:TL branching} If $\delta^2 = 2+q+q^{-1}$ for $q\in \mathbb C$ not a root of unity,
\[ [\Res_{\TL_m\otimes \TL_n}^{\TL_{m+n}} \TL_{m+n}(\lambda), \TL_m(\mu)\otimes \TL_n(\nu)] = c_{\mu\nu}^{\lambda}\]
Here $c$ denotes the Littlewood-Richardson coefficients.
\end{thm}

\begin{proof}
Under the assumption that $\delta^2 = 2+q+q^{-1}$ and $q$ is not a root of unity there is an algebra surjection from the Iwahori-Hecke algebra $H_{n,q}$ (of type A) to $\TL_n$. (cf.\ \cite[Cor.\ 2.11.2]{GdHJ}) It is also known, that the Iwahori-Hecke algebras have the same branching rules as the symmetric groups. (cf.\ \cite[Thm.\ 2.10.9]{GdHJ}) Now the branching rule follows from the branching rule for the symmetric groups.
\end{proof}

We will from now on only consider Temperley--Lieb algebras that satisfy the conditions of \autoref{thm:TL repthy} and \autoref{thm:TL branching}.

%

Let us recall the semisimple representation theory of the Brauer algebras. 

\begin{thm}[{\cite[3.2+3.3]{Wen}+\cite[8D]{Br}}]\label{thm:Br rep thy}
Let  $\delta\in \mathbb C\setminus \{0,1,\dots, n-2\}$. Then $\Br_n(\delta, \mathbb C)$ is semisimple and decomposes into full matrix algebras over $\mathbb C$. Let $\Pi_{\Br_n}$ be the set of all partitions of $n, n-2,n-4, \dots$, then the nonisomorphic simple $\Br_n$-modules are index over $\Pi_{\Br_n}$ and we shall denote them by $\Br_n(\lambda)$ for $\lambda \in \Pi_{\Br_n}$.
\end{thm}

We will henceforth assume that  the Brauer algebras are as considered in the theorem. We are interested in the branching rules of the Brauer algebras. We can extract them from the branching rules of the orthogonal groups $\Orth_n$ over $\mathbb C$ via the following Schur-Weyl duality.

\begin{thm}[{\cite[Section V.5]{Weyl}}] \label{thm:Br semisimple}
Let $\pi \colon \Orth_n \to \GL_n$ be the standard representation, consider $\pi^{\otimes f}\colon \Orth_n \to \GL_{n\cdot f} = \GL( (\mathbb C^n)^{\otimes f})$, then if $n \ge 2f$, $\Br_f(\delta,\mathbb C)$ with $\delta = n$ is isomorphic to $\End_{\Orth_n}((\mathbb C^n)^{\otimes f})$.
\end{thm}

\begin{cor}[Double commutant theory]
If $n\ge 2f$, the $\mathbb C\Orth_n \otimes \Br_f(n,\mathbb C)$-module $(\mathbb C^n)^{\otimes f}$ is semisimple and splits into
\[ (\mathbb C^n)^{\otimes f} \cong \bigoplus_{\lambda\in \Pi_{\Br_f}} \Orth_n(\lambda) \otimes \Br_f(\lambda),\]
where $\Orth_n(\lambda)$  are the irreducible $\Orth_n$-representations of weight $\lambda$.
\end{cor}

The following theorem gives the stable branching rules for the orthogonal groups. Here $\ell(\lambda)$ is the length of the partition of $\lambda$. From this we will immediately get the branching rules for the Brauer algebras.

\begin{thm}[{\cite[2.1.2]{HTW}}]
Given nonnegative integer partitions $\lambda, \mu, \nu$ such that $\ell(\lambda) \le \lfloor n/2\rfloor$ and $\ell(\mu) + \ell(\nu) \le \lfloor n/2\rfloor$, then
\[ [\Orth_n(\mu) \otimes \Orth_n(\nu), \Orth_n(\lambda)] =  \sum_{\alpha, \beta, \gamma} c_{\alpha \beta}^{\lambda} c_{\alpha \gamma}^{\mu} c_{\beta\gamma}^\nu =: d_{\mu\nu}^\lambda.\]
\end{thm}

\begin{cor}\label{cor:Brauerbranching} Let $\delta\in \mathbb C\setminus \{0,1,\dots, e+f-2\}$. Then
\[ [ \Res_{\Br_e\otimes \Br_f}^{\Br_{e+f}} \Br_{e+f}(\lambda),  \Br_{e}(\mu) \otimes \Br_{f}(\nu)] 
= d_{\mu\nu}^\lambda=  \sum_{\alpha, \beta, \gamma} c_{\alpha \beta}^{\lambda} c_{\alpha \gamma}^{\mu} c_{\beta\gamma}^\nu .\]
\end{cor}

\begin{proof}
We may assume that  $\delta = n\ge 2(e+f)$. Then 
\[(\mathbb C^n)^{\otimes(e+f)} = \bigoplus_{\lambda\in \Pi_{\Br_{e+f}}} \Orth_n(\lambda) \otimes \Br_{e+f}(\lambda).\]
Restricting this $\mathbb C\Orth_n \otimes \Br_{e+f}(n)$-module to $\mathbb C\Orth_n \otimes \Br_e(n) \otimes \Br_f(n)$ we get
\begin{align*}
(\mathbb C^n)^{\otimes e}\otimes (\mathbb C^n)^{\otimes f} = &\left(\bigoplus_{\mu\in \Pi_{\Br_e}} \Orth_n(\mu) \otimes \Br_{e}(\mu)\right) \otimes \left(\bigoplus_{\nu\in \Pi_{\Br_f}} \Orth_n(\nu) \otimes \Br_{f}(\nu)\right)\\
= & \bigoplus_{\mu,\nu} \big(\Orth_n(\mu)\otimes \Orth_n(\nu)\big) \otimes \big(\Br_{e}(\mu) \otimes \Br_{f}(\nu)\big).
\end{align*}
Since $\ell(\mu),\ell(\nu)\le e+f \le \lfloor n/2\rfloor$, we have
\[ \Orth_n(\mu)\otimes \Orth_n(\nu)\cong \bigoplus_\lambda \Orth_n(\lambda)^{\oplus d_{\mu\nu}^{\lambda}}.\]
Thus
\[ \Res_{\mathbb C\Orth_n\otimes \Br_e(n) \otimes \Br_f(n)}^{\mathbb C\Orth_n \otimes \Br_{e+f}(n)}\Orth_n(\lambda) \otimes \Br_{e+f}( \lambda)  \cong \Orth_n(\lambda) \otimes \bigoplus_{\mu,\nu} (\Br_{e}(\mu)\otimes \Br_{f}(\nu))^{\oplus d_{\mu\nu}^\lambda},\]
and hence
\[ [ \Res_{\Br_e\otimes \Br_f}^{\Br_{e+f}} \Br_{e+f}(\lambda),  \Br_{e}(\mu) \otimes \Br_{f}(\nu)] 
= d_{\mu\nu}^\lambda.\qedhere\]
\end{proof}

Finally, we introduce the semisimple representation theory of the partition algebras. The results can be found in \cite{M96} although some they might be summarized using earlier results from Martin.

\begin{thm}
Let $\delta \in \mathbb C\setminus \{ 0, \dots, 2n-2\}$. Then $P_n(\delta, \mathbb C)$ is semisimple and decomposes into full matrix algebras over $\mathbb C$. Let $\Pi_{P_n}$ be the set of all partitions of $0, \dots, n$, then the nonisomorphic simple $P_n$-modules are indexed over $\Pi_{P_n}$ and we shall denote them by $P_n(\lambda)$ for $\lambda \in \Pi_{P_n}$.
\end{thm}

From now on we will assume that the partition algebras are as considered in the theorem.

\begin{thm}
Let $\pi \colon \mathfrak S_n \to \GL_n$ be the permutation representation, consider $\pi^{\otimes f}\colon \mathfrak S_n \to \GL_{n\cdot f} = \GL( (\mathbb C^n)^{\otimes f})$, then if $n \ge 2f$, $P_f(\delta,\mathbb C)$ with $\delta = n$ is isomorphic to $\End_{\mathfrak S_n}((\mathbb C^n)^{\otimes f})$.
\end{thm}

\begin{cor}[Double commutant theory]
If $n\ge 2f$, the $\mathbb C\mathfrak S_n \otimes P_f(n,\mathbb C)$-module $(\mathbb C^n)^{\otimes f}$ is semisimple and splits into
\[ (\mathbb C^n)^{\otimes f} \cong \bigoplus_{\lambda\in \Pi_{P_f}} \mathfrak S_n(\lambda[n]) \otimes P_f(\lambda),\]
where $\mathfrak S_n(\lambda)$  irreducible $\mathfrak S_n$-representations and $\lambda[n] = (n-|\lambda|,\lambda_1, \lambda_2, \dots)$.
\end{cor}

With the same proof as for  \autoref{cor:Brauerbranching}, we can now find the branching rules for the partition algebras. The coefficients are called \emph{reduced Kronecker coefficients}.

\begin{cor}
Let $\delta\in \mathbb C\setminus \{0,1,\dots, 2(e+f)-2\}$. Then
\[
 [\Ind_{P_e\otimes P_f}^{P_{e+f}} P_{e}(\mu)\otimes P_{f}(\nu) , P_{e+f}(\lambda)] = \bar g_{\lambda,\mu,\nu} := [ \mathfrak S_n(\mu[n]) \otimes \mathfrak S_n(\nu[n]),  \mathfrak S_n(\lambda[n])].
\]
\end{cor}

\subsection{Definitions of representation stability}

In \cite{CF}, Church and Farb originally defined for every partition $\lambda$ and every $n \ge |\lambda|+\lambda_1$ a partition $\lambda[n]$ of $n$ by simply adding a first row of boxes which is long enough. I.e.\ if $\lambda = (\lambda_1,\dots, \lambda_{\ell(\lambda)})$ then $\lambda[n] = (n-|\lambda|, \lambda_1, \dots, \lambda_{\ell(\lambda)})$. An $\FI$-module $V$ is then called representation stable if it satifies the following three conditions.
\begin{description}
\item[Injectivity] The canonical map $\phi\colon V_n \to V_{n+1}$ is injective for all large enough $n\in\mathbb N$.
\item[Surjectivity] The induced map $\Ind_{\mathfrak S_n}^{\mathfrak S_{n+1}} \phi\colon \Ind_{\mathfrak S_n}^{\mathfrak S_{n+1}}V_n \to V_{n+1}$ is surjective for all large enough $n\in \mathbb N$.
\item[Multiplicity stability] If we write
\[ V_n \cong \bigoplus_{\lambda} \mathfrak S_n(\lambda[n])^{\oplus c_{\lambda,n}}\]
then $c_{\lambda,n}$ is independent of $n$ for all large enough $n\in \mathbb N$. 
\end{description}

For $\CA$-modules we easily find analogues of the conditions injectivity and surjectivity. For multiplicity stability we make the following definitions. Let $V$ be a $\CTL$-module, then we can write
\[ V_n \cong \bigoplus_{\lambda} \TL_n({\lambda[n]})^{\oplus c_{\lambda,n}}\]
with partitions $\lambda$ with $\ell(\lambda) \le 1$. We say that $V$ is multiplicity stable if $c_{\lambda,n}$ is independent of $n$ for all large enough $n\in \mathbb N$. Let $V$ be a $\CBr$-module, then we can write
\[ V_n \cong \bigoplus_{\lambda} \bigoplus_{i\le \lfloor\frac n2\rfloor} \Br_n{(\lambda[n-2i])}^{\oplus c_{\lambda,i,n}}.\]
We say that $V$ is multiplicity stable if $c_{\lambda,i,n}$ is independent of $n$ for all large enough $n\in \mathbb N$. Let $V$ be a $\CP$-module, then we can write
\[ V_n \cong  \bigoplus_{\lambda\in \Pi_{P_n}} P_n{(\lambda)}^{\oplus c_{\lambda,n}}.\]
We say that $V$ is multiplicity stable if $c_{\lambda,n}$ is independent of $n$ for all large enough $n\in \mathbb N$. 

\begin{Def}\label{def:repstab}
We call a $\CA$-module \emph{representation stable} if it satifies injectivity, surjectivity, and multiplicity stability.
\end{Def}

Next we will introduce the stability degree of $\CA$-modules.  We first make the observation that because of the embedding \eqref{eq:embedding} the tensor product $\mathbb C\otimes_{A_{n-m}} V_n$ is an $A_{m}$-module for any $A_n$-module $V_n$. Furthermore $\phi\colon V_n \to V_{n+1}$ induces an $A_m$-map
\[\xymatrix{ \displaystyle \mathbb C \tens[A_{n-m}] V_n \ar[r] & \displaystyle \mathbb C\tens[A_{n-m}] V_{n+1} \ar@{->>}[r] & \displaystyle \mathbb C\tens[A_{n+1-m}] V_{n+1}.}\]
Here we use the inclusion $A_{n-m}\otimes A_{1} \inject A_{n+1-m}$. 

\begin{Def}
Let $\tau_{n,a}$ be the functor $\tau_{n,a} V_n = \mathbb C \otimes_{A_{n-a}} V_n$ from $A_n$-modules to $A_a$-modules. We say a $\CA$-module has injectivity degree, surjectivity degree, or stability degree $\le s$ if the maps
\[ \xymatrix{ \tau_{n+a,a} V_{n+a} \ar[r]^<<<<<{\phi_*} & \tau_{n+a+1,a} V_{n+a+1}} \]
is injective, surjective, or bijective (resp.) for all nonnegative integers $a$ and all $n\ge s$.
\end{Def}

\begin{rem}\label{rem:tau}
Note that if
\[ \Res_{A_{a}\otimes A_{n-a}}^{A_n} V_n \cong \bigoplus W_i\otimes W'_i,\]
with simple $A_{n-a} \otimes A_a$-modules $W_i\otimes W'_i$, then
\[ \tau_{n,a}V_n \cong \bigoplus_{W'_i \text{ trivial}} W_i.\]
\end{rem}

\begin{thm}\label{thm:stabdegM(m)}
Let $M(m)$ be the $\CA$-module defined by $M(m)_n = \Hom_{\CA}(m,n)$. Then $M(m)$ has injectivity degree $\le0$, for $\CA\in \{\CTL,\CBr\}$ surjectivity degree $\le m$, and for $\CA=\CP$ surjectivity degree $\le 2m$.
\end{thm}

\begin{proof}
We need to investigate
\[ \xymatrix{ \tau_{n+a,a} M(m)_{n+a} \ar[r]^<<<<<{\phi_*} & \tau_{n+a+1,a} M(m)_{n+a+1}}. \]
We first will find a basis of
\[ \tau_{n+a,a} M(m)_{n+a} = \mathbb C \tens[A_n] A_{n+a} \tens[A_{n+a-m}] \mathbb C.\]
Similar to the basis of $A_n\otimes_{A_{n-m}} \mathbb C$, we can describe the basis elements by diagrams. In this case we have two blobs, an $n$-blob on the LHS and an $(n+a-m)$-blob on the RHS. The restrictions for the blobs are exactly as before. Concretely this means in each block of the partition each blob can appear at most once, an $x$-blob must appear in exactly $x$ partitions, and a blob cannot appear in a singleton. The map $\phi_*$ can be described by adding a partition that exactly consists both blobs. 

\begin{figure}[ht]
\centering
\[
\begin{tikzpicture}[x=1.5cm,y=-.5cm,baseline=-1.05cm]

\node[b] (a1) at (0,0) {$2$};

\node[v] (a3) at (0,1) {};
\node[v] (a4) at (0,2) {};
\node[v] (a5) at (0,3) {};

\node[b] (b1) at (1,0) {$3$};
\node[v] (b4) at (1,2) {};
\node[v] (b5) at (1,3) {};

\draw[e] (a1) to[out=0, in=180] (b1);
\draw[e] (a1) to[out=-20, in=0]   (a4);
\draw[e] (b1) to[out=220, in=180] (b4);
\draw[e] (a5) to[out=0, in=0] (a3);
\draw[e] (b1) to[out=200,in=0] (a3);
\draw[e] (b5) to[out=180,in=270] (.5,0);

\end{tikzpicture}
\qquad
\longmapsto\qquad
\begin{tikzpicture}[x=1.5cm,y=-.5cm,baseline=-1.05cm]
\node[b] (a1) at (0,0) {$3$};

\node[v] (a3) at (0,1) {};
\node[v] (a4) at (0,2) {};
\node[v] (a5) at (0,3) {};

\node[b] (b1) at (1,0) {$4$};
\node[v] (b4) at (1,2) {};
\node[v] (b5) at (1,3) {};

\draw[e] (a1) to[out=0, in=180] (b1);
\draw[e] (a1) to[out=20, in=160] (b1);
\draw[e] (a1) to[out=-20, in=0]   (a4);
\draw[e] (b1) to[out=220, in=180] (b4);
\draw[e] (a5) to[out=0, in=0] (a3);
\draw[e] (b1) to[out=200,in=0] (a3);
\draw[e] (b5) to[out=180,in=270] (.5,0);

\end{tikzpicture}
\]
\caption[$\phi_*$]{$\begin{array}{l}\phantom{=}\phi_*( 1 \tens[P_{2}] \{ \{-5,1,5\},\{-4,2\}, \{-3,-1,4\},\{2,3\}\} \tens[P_3] 1 ) \\ = 1 \tens[P_{3}] \{\{-6,6\}, \{-5,1,5\},\{-4,2\}, \{-3,-1,4\},\{2,3\}\} \tens[P_4] 1\end{array}$}
\label{fig:Ex taublob}
\end{figure}

$\phi_*$ is then always injective. For the Temperley--Lieb and the Brauer case, assume that $d\in \tau_{n+a+1,a} M(m)_{n+a+1}$ is a diagram that is not in the image of $\phi_*$. That means there is no edge between the two blobs. Thus the $(n+1)+ (n+a+1-m)$ edges coming out of the blobs must be connected to one of the $a+m$ dots each. This gives the inequality
\[ (n+1)+ (n+a+1-m) = 2n +2 +a - m \le a+m \iff 2n +2 \le 2m \iff n<  m.\]
In particular $\phi_*$ is surjective for $n \ge m$.

For the partition algebra case, if $d$ is not in the image of $\phi_*$, there is no $2$-block that connects the two blobs. Thus each block that contains a blob must also connected to a dot. Therefore there must be at least $n+a+1-m$ dots, one for every block that contains the $(n+a+1-m)$-blob on the RHS. This implies the inequality
\[ n+a+1-m \le a+m \iff n +1 \le 2m \iff n < 2m .\]
In particular $\phi_*$ is surjective for $n \ge 2m$.
\end{proof}

\subsection{From stability degree to representation stability}

\begin{prop}\label{prop:TLtau}
Let $m\in \mathbb N$ and $\lambda = (m)$. The following statements are true for all $n\ge 2m$.
\begin{enumerate}
\item $\tau_{n,a}\TL_n({\lambda[n]}) = 0$ if  $a < m$
\item $\tau_{n,m}\TL_n({\lambda[n]}) \cong \TL_n({\lambda})$
\item For fixed $a\in \mathbb N$ the $\TL_a$-modules $\tau_{n,a}\TL_n({\lambda[n]})$ are independent of $n\ge a+m$.
\end{enumerate}
\end{prop}

\begin{proof}
Because of  \autoref{rem:tau} and the branching rule in  \autoref{thm:TL branching}
\[ [\TL_a(\mu),\tau_{n,a} \TL_n(\lambda[n])] = c_{\mu,(n-a)}^{\lambda[n]}.\]
Pieri's formula (see \cite[(A.7)]{FH}) says that
\[ c^{\lambda}_{\mu,(a)} = \begin{cases} 1 & \text{ if $\mu$ can be obtained from $\lambda$ be removing $a$ boxes but at most one per column}\\ 0 &\text{ otherwise.}\end{cases}\]
So to be nonzero, it must be
\[ n-m \ge n-a \Longleftrightarrow a \ge m.\]
This proves (a). Assuming $a = m$, we also see that $\mu = (m) = \lambda$. This proves (b). If $n-a\ge m$ there are exactly $\min(m, \lfloor a/2\rfloor) + 1$ many different partitions that can be obtained from $\lambda[n]$ by removing $n-a$ boxes but at most one per column. These are
\[ (a), (a-1,1), \dots, (a- \min(m, \lfloor a/2\rfloor) , \min(m, \lfloor a/2\rfloor) ).\] 
This proves (c).
\end{proof}

\begin{prop}\label{prop:Brauertau} Let $\mu\in \Pi_{\Br_m}$ and $|\mu|+\mu_1 \le n -2i$, then the following statements hold.
\begin{enumerate}
\item $[ \Br_m(\lambda), \tau_{n,m}\Br_n(\mu[n-2i])]= 0$ if $|\lambda| < m-2i$.
\item $[ \Br_m(\lambda), \tau_{n,m}\Br_n(\mu[n-2i])]= 0$ if $|\lambda| = m-2i$ and $|\mu|> m-2i$.
\item $[ \Br_m(\lambda), \tau_{n,m}\Br_n(\mu[n-2i])]= \delta_{\lambda \mu}$ if $|\lambda| =|\mu|= m-2i$.
\item $\tau_{n,m}\Br_n(\mu[n-2i])$ is independent of $n\ge  m+\mu_1+i$.
\end{enumerate}
\end{prop}

\begin{proof}
Because of  \autoref{rem:tau} and the branching rule in  \autoref{cor:Brauerbranching}
\[ [ \Br_m(\lambda), \tau_{n,m}\Br_n(\mu[n-2i])] = d_{(n-m),\lambda}^{\mu[n-2i]} = \sum_{\alpha, \beta, \gamma} c_{\alpha \beta}^{\mu[n-2i]} c_{\alpha \gamma}^{\lambda} c_{\beta\gamma}^{(n-m)}.\]
The Littlewood-Richardson coefficient $c_{\beta\gamma}^{(n-m)}=1$ if $\beta=(n-m-l)$ and $\gamma =(l)$ for some $0\le l \le n-m$ and it is zero otherwise. In other words
\[ d_{(n-m),\lambda}^{\mu[n-2i]} = \sum_{\alpha, 0\le l \le n-m} c_{\alpha,(n-m-l)}^{\mu[n-2i]} c_{\alpha, (l)}^{\lambda}.\]
We also calculate
\[ |\lambda| = |\alpha| + l = (n-2i) -(n-m-l) +l = m+2l-2i \ge m-2i.\]
This proves (a). 

Now assuming that $|\lambda|= m-2i$, we know $l =0$ and $\alpha = \lambda$. That means  
\[d_{(n-m),\lambda}^{\mu[n-2i]} =  c_{\lambda,(n-m)}^{\mu[n-2i]}.\]
From Pieri's formula we know that $c_{\lambda,(n-m)}^{\mu[n-2i]}=0$ unless the first entry of $\mu[n-2i]$ which is $n-2i-|\mu|$ is at least $n-m$. In other words $|\mu| \le m-2i$. This proves (b).

Using Pieri's formula again we see that $c_{\lambda,(n-m)}^{\mu[n-2i]}= \delta_{\mu\lambda}$ if $n-2i-|\mu| = n-m$. This proves (c).

For (d), say $|\lambda| = m +2l -2i$ (this implies $l \le i$), then
\[ d_{(n-m),\lambda}^{\mu[n-2i]} = \sum_{\alpha} c_{\alpha,(n-m-l)}^{\mu[n-2i]} c_{\alpha, (l)}^{\lambda}.\]
The set of partitions $\alpha$ such that $c_{\alpha, (l)}^{\lambda}\neq 0$ is completely independent of $n$. From Pieri's formula $c_{\alpha,(n-m-l)}^{\mu[n-2i]}$ is fixed for $n-m-l \ge \mu_1$. Thus $d_{(n-m),\lambda}^{\mu[n-2i]} $ is independent of
\[ n \ge m+ \mu_1 + i \ge m + \mu_1 + l.\qedhere\]
\end{proof}

\begin{prop}\label{prop:partition tau}
Let $\lambda \in \Pi_{P_n}$ then for all $n\ge a$
\[ \tau_{n,a} P_n(\lambda) \cong \begin{cases} P_a(\lambda) & \text{ if $\lambda \in \Pi_{P_a}$}\\ 0 &\text{ otherwise.}\end{cases}\]
\end{prop}

\begin{proof}
The multiplicity of $P_a(\mu)$ in $\tau_{n,a} P_n(\lambda)$ is given by the reduced Kronecker coefficient 
\[ \bar g_{\emptyset, \mu}^{\lambda} = [ \mathfrak S_{n}(\emptyset[2n]) \otimes \mathfrak S_{2n}(\mu[2n]), \mathfrak S_{2n}(\lambda[2n])].\]
But this is $\delta_{\mu[2n],\lambda[2n]}= \delta_{\mu,\lambda}$ because $\emptyset[2n] = (2n)$ which indexes the trivial $\mathfrak S_{2n}$-representation.
\end{proof}

\begin{prop}\label{prop:finpres implies stabdeg}
Let $V$ be a finitely presented $\CA$-module generated in degree $g$ and with relations generated in degree $r$, then $V$ has stability degree $\le \max(\bar g,\bar r)$, where $\bar g = g$ and $\bar r = r$ for $\CA \in \{\CTL, \CBr\}$ and $\bar g = 2g$ and $\bar r = 2r$ for $\CA= \CP$.
\end{prop}

\begin{proof}
Let 
\[ \xymatrix{
0 \ar[r] & K \ar[r] & P \ar[r] & V \ar[r] & 0
}\]
be a short exact sequence where $P$ is direct sum of $M(m)$'s with $m\le g$ and $K$ is a  $\CA$-module finitely generated in degree $r$. Because tensoring is exact in the semisimple setting, we get the following exact commutative diagram.
\[ \xymatrix{
0 \ar[r] & \tau_{n+a,a} K_{n+a} \ar[r]\ar[d]^{\phi_K} &\tau_{n+a,a} P_{n+a} \ar[r]\ar[d]^{\phi_P} & \tau_{n+a,a}V_{n+a} \ar[r]\ar[d]^{\phi_V} & 0\\
0 \ar[r] & \tau_{1+n+a,a} K_{1+n+a} \ar[r] &\tau_{1+n+a,a} P_{1+n+a} \ar[r] & \tau_{1+n+a,a}V_{1+n+a} \ar[r] & 0
}\]
We know from  \autoref{thm:stabdegM(m)} that $\phi_P$ is injective for all $n\ge0$ and surjective for all $n\ge \bar g$. From the Four Lemma, it is immediate that also $\phi_V$ is surjective for $n \ge \bar g$. Similarly $\phi_K$ is surjective for $n \ge \bar r$. Using the Four Lemma again, we see that $\phi_V$ is injective for $n \ge \bar r$. Thus it is bijective for $n\ge \max(\bar g, \bar r)$, which prove the assertion.
\end{proof}

\begin{prop}\label{prop:TLweight}
Let $V$ be a $\CTL$-module generated in degree $g$ and $\TL_n(\lambda[n])$ is a constituent of $V_n$, then $|\lambda|  \le g$.
\end{prop}

\begin{proof}
It suffices to prove that if $\TL_n(\lambda[n])$ is a constituent of $M(m)_n$, then $\lambda \le m$.

Recall that 
\[M(m)_n = \TL_n \otimes_{\TL_{n-m}} \mathbb C = \Ind_{\TL_m\otimes \TL_{n-m}}^{\TL_n} \TL_m\otimes \mathbb C.\]
Thus we need to consider the constituents $\TL_n(\lambda[n])$ of
\[ \Ind_{\TL_m\otimes \TL_{n-m}}^{\TL_n} \TL_m(\mu)\otimes \TL_{n-m}(n-m)\]
for all suitable partitions $\mu$. From the branching rule we see that
\[ [\TL_n(\lambda[n]), \Ind_{\TL_m\otimes \TL_{n-m}}^{\TL_n} \TL_m(\mu)\otimes \TL_{n-m}(n-m)]  = c_{(n-m),\mu}^{\lambda[n]}. \]
From Pieri's formula we see that $n-|\lambda| \ge n-m$ for this coefficient to be nonzero. Thus $|\lambda|\le m$.
\end{proof}

\begin{prop}\label{prop:Brauerweight}
Let $V$ be a $\CBr$-module generated in degree $g$ and $\Br_n(\lambda[n-2i])$ is a constituent of $V_n$, then $|\lambda| + i \le g$.
\end{prop}

\begin{proof}
It suffices to prove that if $\Br_n(\lambda[n-2i])$ is a constituent of $M(m)_n$, then $\lambda + i \le m$.

Recall that 
\[M(m)_n = \Br_n \otimes_{\Br_{n-m}} \mathbb C = \Ind_{\Br_m\otimes \Br_{n-m}}^{\Br_n} \Br_m\otimes \mathbb C.\]
Thus we need to consider the constituents $\Br_n(\lambda[n-2i])$ of
\[ \Ind_{\Br_m\otimes \Br_{n-m}}^{\Br_n} \Br_m(\mu)\otimes \Br_{n-m}(n-m)\]
for all suitable partitions $\mu$. From the branching rule we see that
\begin{multline*}
 [\Br_n(\lambda[n-2i]), \Ind_{\Br_m\otimes \Br_{n-m}}^{\Br_n} \Br_m(\mu)\otimes \Br_{n-m}(n-m)] \\= d_{(n-m),\mu}^{\lambda[n-2i]} = \sum c_{\alpha, (n-m-l)}^{\lambda[n-2i]}c_{\alpha, (l)}^{\mu}.
 \end{multline*}
The last term was calculated in the proof of  \autoref{prop:Brauertau}. There we also saw that
\[ |\mu| = m + 2l -2i \stackrel{m\ge |\mu|}{\implies} i\ge l.\]
With Pieri's formula we see that the first term of $\lambda[n-2i]$ which is $n-2i -|\lambda|$ must be at least $n-m-l$, i.e.
\[ |\lambda|+i \le m +l -i \le m.\qedhere\]
\end{proof}

\begin{prop}\label{prop:partition weight}
Let $V$ be a $\CP$-module generated in degree $g$ and $P_n(\lambda)$  is a constituent of $V_n$, then $|\lambda| \le g$.
\end{prop}

\begin{proof}
Again it suffices to prove that if $P_n(\lambda)$ is a constituent of $M(m)_n$, then $|\lambda| \le m$.

Here we need to investigate the constituents of
\[ \Ind_{P_m\otimes P_{n-m}}^{P_n} P_m(\mu)\otimes P_{n-m}(\emptyset)\]
for all partitions $\mu$ of size at most $m$. But as in  \autoref{prop:partition tau}, we find
\[ [P_n(\lambda),  \Ind_{P_m\otimes P_{n-m}}^{P_n} P_m(\mu)\otimes P_{n-m}(\emptyset)] = \delta_{\lambda \mu}.\]
This shows that $|\lambda| \le m$. 
\end{proof}

The following theorem proves \autoref{thmA:TL}.

\begin{thm}
Let $V$ is a finitely presented $\CTL$-module generated in degree $g$ with relations generated in degree $r$, then $V$ is representation stable with a stable range $\ge g + \max(g,r)$.
\end{thm}

\begin{proof}
Let us write
\[ V_n = \bigoplus_{\lambda} \TL_{\mu[n]}^{\oplus c_{\mu,n}}. \]
From  \autoref{prop:TLweight}, we know $|\mu|  \le g$. We will prove that $c_{\mu,n}$ is independent of $n\ge g + \max(g,r)$ by induction over $|\mu|$. Since for small enough $|\mu|$ there is nothing to prove, it suffices to provide the induction step. Fix $\lambda$  and assume the assertion has been proved for all $|\mu|<|\lambda|$.

Set $m = |\lambda|  \le g$. We will count the multiplicity of $\TL_m(\lambda)$ in $\tau_{n,m} V_n$.
\begin{align}
&\big[ \TL_m(\lambda)\,,\,\tau_{n,m} V_n\big]  \label{eq:TLtau1}\\
= &\hspace{1.3em} \sum_{|\mu|>|\lambda|} \big[ \TL_m(\lambda)\,,\,\tau_{n,m} \TL_n(\mu[n]) \big]\cdot  c_{\mu,n}\label{eq:TLtau2}\\
&+ \hspace{.3em} \sum_{|\mu| = |\lambda|} \big[ \TL_m(\lambda)\,,\,\tau_{n,m} \TL_n(\mu[n]) \big]\cdot c_{\mu,n}\label{eq:TLtau3}\\
&+ \sum_{|\mu|< |\lambda|} \big[ \TL_m(\lambda)\,,\,\tau_{n,m} \TL_n(\mu[n]) \big] \cdot  c_{\mu,n}.\label{eq:TLtau4}
\end{align}
From  \autoref{prop:finpres implies stabdeg} we know that \eqref{eq:TLtau1} is independent of \[n\ge m +\max(g,r).\]  \hyperref[prop:TLtau]{\autoref{prop:TLtau}(a)} says that \eqref{eq:TLtau2} is zero. \hyperref[prop:TLtau]{\autoref{prop:TLtau}(c)} and the induction hypothesis provide that \eqref{eq:TLtau4} is independent of \[ n \ge \max\bigg( m+ |\mu| \,,\, g + \max(g,r) \bigg). \]
Finally  \hyperref[prop:TLtau]{\autoref{prop:TLtau}(b)} says that \eqref{eq:TLtau3} is $c_{\lambda,n}$. 

Because $m \le g$ and $|\mu| \le g$, we get that $c_{\lambda,n}$ is independent of $n \ge  g + \max(g,r)$.
\end{proof}

The following theorem proves \autoref{thmA:Br}.

\begin{thm}
Let $V$ is a finitely presented $\CBr$-module generated in degree $g$ with relations generated in degree $r$, then $V$ is representation stable with a stable range $\ge 2g + \max(g,r)$.
\end{thm}

\begin{proof}
Let us write
\[ V_n = \bigoplus_{\mu,i} \Br_n(\mu[n-2i])^{\oplus c_{\mu,i,n}}.\]
From  \autoref{prop:Brauerweight}, we know $|\mu| + i \le g$. We will prove that $c_{\mu,i,n}$ is independent of $n\ge 2g + \max(g,r)$ by induction first over $i$ going down, then over $|\mu|$ going up. Since for large enough $i$ and small enough $|\mu|$ there is nothing to prove, it suffices to provide the induction step. Fix $\lambda$ and $j$ and assume the assertion has been proved for all $i>j$ and for all $|\mu|<|\lambda|$ if $i=j$.

Set $m = |\lambda| + 2i \le 2g$. We will count the multiplicity of $\Br_m(\lambda)$ in $\tau_{n,m} V_n$.
\begin{align}
&\big[ \Br_m(\lambda)\,,\,\tau_{n,m} V_n\big]  \label{eqBrtauV1}\\
= &\hspace{1.3em} \sum_{\mu, i>j} \big[ \Br_m(\lambda)\,,\,\tau_{n,m} \Br_n(\mu[n-2i]) \big]\cdot  c_{\mu,i,n}\label{eqBrtauV2}\\
&+ \hspace{.3em} \sum_{\mu, i<j} \big[ \Br_m(\lambda)\,,\,\tau_{n,m} \Br_n(\mu[n-2i]) \big]\cdot c_{\mu,i,n}\label{eqBrtauV3}\\
&+ \sum_{|\mu|> |\lambda|} \big[ \Br_m(\lambda)\,,\,\tau_{n,m} \Br_n(\mu[n-2j]) \big] \cdot  c_{\mu,j,n}\label{eqBrtauV4}\\
&+ \sum_{|\mu|< |\lambda|} \big[ \Br_m(\lambda)\,,\,\tau_{n,m} \Br_n(\mu[n-2j]) \big]\cdot  c_{\mu,j,n}\label{eqBrtauV5}\\
&+ \sum_{|\mu|= |\lambda|} \big[ \Br_m(\lambda)\,,\,\tau_{n,m} \Br_n(\mu[n-2j]) \big]\cdot c_{\mu,j,n}.\label{eqBrtauV6}
\end{align}
From  \autoref{prop:finpres implies stabdeg} we know that \eqref{eqBrtauV1} is independent of \[n\ge m +\max(g,r).\] \hyperref[prop:Brauertau]{\autoref{prop:Brauertau}(a)} says that \eqref{eqBrtauV2} is zero.  \hyperref[prop:Brauertau]{\autoref{prop:Brauertau}(d)} and the induction hypothesis provide that \eqref{eqBrtauV3} is independent of \[ n \ge \max\bigg( m+ \mu_1 + i \,,\, 2g + \max(g,r) \bigg). \]
 \hyperref[prop:Brauertau]{\autoref{prop:Brauertau}(b)} says that \eqref{eqBrtauV4} is zero. With  \hyperref[prop:Brauertau]{\autoref{prop:Brauertau}(d)} and the induction hypothesis we see that \eqref{eqBrtauV5} is independent of \[ n \ge \max\bigg( m+ \mu_1 + j \,,\, 2g + \max(g,r) \bigg). \]
Finally  \hyperref[prop:Brauertau]{\autoref{prop:Brauertau}(b)} says that \eqref{eqBrtauV6} is $c_{\lambda,j,n}$. 

Because $m \le 2g$ and $\mu_1+i \le g$, we get that $c_{\lambda,j,n}$ is independent of $n \ge  2g + \max(g,r)$.
\end{proof}

The following theorem proves \autoref{thmA:P}.

\begin{thm}
Let $V$ is a finitely presented $\CP$-module generated in degree $g$ with relations generated in degree $r$, then $V$ is representation stable with a stable range $\ge 2g + \max(2g,2r)$.
\end{thm}

\begin{proof}
Let us write
\[ V_n = \bigoplus_{\lambda} P_n(\lambda)^{\oplus c_{\lambda,n}}.\]
From  \autoref{prop:partition weight}, we know $|\lambda|  \le g$. Together with  \autoref{prop:partition tau}, we know
\[ \tau_{n,g} V_n = \bigoplus_{\lambda} P_g(\lambda)^{\oplus c_{\lambda,n}}.\]
Thus 
\[ c_{\lambda,n} = [ P_g(\lambda), \tau_{n,g} V_n].\]
Because of  \autoref{prop:finpres implies stabdeg} we know that $\tau_{n,g} V_n$ is independent of $n\ge 2g+\max(2g,2r)$, and thus so is $c_{\lambda,n}$.
\end{proof}

\bibliographystyle{halpha}
\bibliography{repstab}

\end{document}